\documentclass[a4paper,11pt]{amsart}
\usepackage{amstext,amsfonts,amsthm,graphicx,amssymb,amscd,epsfig}
\usepackage{amsmath}
\usepackage[ansinew]{inputenc} 
\usepackage{psfrag}
\usepackage{mathrsfs}
\usepackage[active]{srcltx}
\newcommand{\inner}[2]{\left\langle{#1},{#2}\right\rangle}
\numberwithin{equation}{section}

\theoremstyle{definition}
\newtheorem{definition}{Definition}[section]
\newtheorem{ex}[definition]{Example}
\newtheorem{rem}[definition]{Remark}

\theoremstyle{plain}
\newtheorem{prop}[definition]{Proposition}
\newtheorem{lem}[definition]{Lemma}
\newtheorem{coro}[definition]{Corollary}
\newtheorem{teo}[definition]{Theorem}

\newfont{\bbb}{msbm10 scaled\magstephalf} 

\def\R{\mathbb R}

\def\R{\mbox{\bbb R}}

\def\A{\mathscr A}

\usepackage[usenames]{color}

\title{The axial curvature for corank 1 singular surfaces}

\author{R. Oset Sinha, K. Saji}

\date{\today}

\address{Departament de Matem\'atiques,
Universitat de Val\`encia, Campus de Burjassot, 46100 Burjassot,
Spain}

\email{raul.oset@uv.es}

\address{Department of Mathematics,
Kobe University, Rokko 1-1, Nada,
Kobe 657-8501, Japan}

\email{saji@math.kobe-u.ac.jp}

\thanks{Work of R. Oset Sinha partially supported by MICINN Grant PGC2018-094889-B-I00}
\thanks{Work of K. Saji supported by JSPS KAKENHI Grant Number JP26400087}

\begin{document}
\begin{abstract}
For singular corank 1 surfaces in $\mathbb R^3$ we introduce a distinguished normal vector called the axial vector. Using this vector and the curvature parabola we define a new type of curvature called the axial curvature, which generalizes the singular curvature for frontal type singularities. We then study contact properties of the surface with respect to the plane orthogonal to the axial vector and show how they are related to the axial curvature. Finally, for certain fold type singularities, we relate the axial curvature with the Gaussian curvature of an appropriate blow up.
\end{abstract}

\maketitle

\section{Introduction}
Singularities of surfaces or singular surfaces in 3-space have been of interest for a very long time. However, in the last 15 years the study of the differential geometry of singular surfaces has seen a huge development due to the growing number of situations in which this type of surfaces appear. In fact, these object are not only cherished by singularity theorists but also by differential geometers. The introduction of singularity theory techniques has been crucial in the development of the area. Papers such as \cite{kokubuetal} or \cite{sajietalannals}, which introduce new types of curvature and study the behaviour of the Gaussian curvature near singular points for wave-fronts, have become seminal papers in the area.

Wave-fronts, or frontals in general, have a well-defined normal vector even at the singular points, so it is in a way easier to study geometrical properties for these kinds of singularities. For different types of singularities such as the cross-cap, which is the only type of singularity a stable map germ $f:(\mathbb R^2,0)\to (\mathbb R^3,0)$ can have, this is not the case. In \cite{MartinsBallesteros}, corank 1 singularities are studied in general (by corank we mean the corank of the differential of the local parametrisation of the surface). The authors define a curvature parabola in the normal plane similar to the curvature ellipse for regular surfaces in $\mathbb R^4$, which encodes all the second order geometry of the surface at the singular point. In the case the parabola is degenerate (when the singular point is not a cross-cap) they define the \emph{umbilic curvature} $\kappa_u$ as the projection of the parabola to a certain distinguished normal direction, which captures degenerate contact with spheres. This curvature generalizes the normal curvature for fronts (\cite{sajietalannals}, \cite{martinssaji1}). In \cite{teramoto} it is shown that the normal curvature is a kind of bounded principal curvature for front singularities, so the umbilic curvature can be seen as a principal curvature for corank 1 singularities.

The idea of obtaining the principal curvatures in a certain normal direction by projections comes from the curvature ellipse of surfaces in $\mathbb R^4$ (\cite{nunoromerobringas}). In fact, in \cite{garciasoto} the authors introduce the concept of lines of axial curvature as the lines of curvature corresponding to the principal curvatures in the normal direction corresponding to the axis of the ellipse. Inspired by these ideas we define in Section 3 the \emph{axial curvature} $\kappa_a$ as the minimum value of the projection of the curvature parabola on the \emph{axial vector} $v_a$, where $v_a$ is the axis of symmetry of the parabola when it is non-degenerate or the direction of the line which contains the parabola when it is degenerate but not a point. We prove that this curvature is intrinsic and give coordinate free expressions for it. In Section 4 we show that it generalizes the singular curvature for fronts (\cite{kokubuetal}). Considering the amount of applications that the singular curvature has in generalizing concepts and results of regular surfaces to frontals, this gives an idea of the potential of the axial curvature.

Section 5 is devoted to the study of the contact of a surface with the plane orthogonal to the axial vector by analyzing the height function in the direction of $v_a$. We characterize the type of contact by the axial curvature and give criteria to distinguish when a singular point is elliptic, hyperbolic or parabolic by looking at the curves of intersection of the surface with the plane orthogonal to $v_a$. Section 2 contains the preliminaries about corank 1 surfaces in $\mathbb R^3$ from \cite{MartinsBallesteros}. Finally, in Section 6, for certain fold type singularities (i.e. $j^{2}f(0)\sim_{\mathcal{A}^{2}}(x,y^{2},0)$) we relate the axial curvature to the Gaussian curvature of an appropriate blow up and we justify why we cannot obtain a good Koenderink type formula due to the appearance of a certain term. As a by-product we prove that this term is an obstruction to frontality.

\section{Preliminaries}\label{section-notation}

We state some definitions and results about corank 1 surfaces in $\mathbb R^3$ (see \cite{MartinsBallesteros} for details). Given a surface
$M\subset\mathbb{R}^{3}$ with a corank 1 singularity at $p\in M$, we shall assume it as the
image of a smooth map $f:\mathbb R^2\rightarrow\mathbb{R}^{3}$, such that $f(q)=p\in M$, where $q$ is a corank 1 singular point of $f$. Notice that we are taking $\tilde M=\mathbb R^2$ and $\varphi=id$ in the construction in \cite{MartinsBallesteros}.

The tangent line to $M$ at $p$ is the set $T_{p}M=\mbox{im}(df_{q})$, where
$df_{q}:T_{q}\mathbb R^2\rightarrow T_{p}\mathbb{R}^{3}$ and the normal plane
$N_{p}M$ satisfies $T_{p}\mathbb{R}^{3}=T_{p}M\oplus N_{p}M$.
The \emph{first fundamental form} $I:T_{q}\mathbb R^2\times T_{q}\mathbb R^2\rightarrow\mathbb{R}$
is given by
$$
I(X,Y)=\langle df_{q}(X),df_{q}(Y)\rangle,\ \forall\ X,Y\in T_{q}\mathbb R^2.
$$
With the parametrisation $f$ and if
$\{\partial_{u},\partial_{v}\}$ is a basis for $T_{q}\mathbb R^2$, the coefficients
of the first fundamental form are:
$$
\begin{array}{c}
E(q)=I(\partial_{u},\partial_{u})=\langle f_{u},f_{u}\rangle(q),\ F(q)=I(\partial_{u},\partial_{v})=\langle f_{u},f_{v}\rangle(q)\\
G(q)=I(\partial_{v},\partial_{v})=\langle f_{v},f_{v}\rangle(q),
\end{array}
$$
and taking $X=a\partial_{u}+b\partial_{v}\in T_{q}\mathbb R^2$,
$I(X,X)=a^{2}E(q)+2abF(q)+b^{2}G(q)$. This induces a pseudometric in $T_{q}\mathbb R^2$.
Let $\perp:T_{p}\mathbb{R}^{3}\rightarrow N_{p}M$, be the orthogonal projection onto
the normal plane. The \emph{second fundamental form} of $M$ at $p$,
$II:T_{q}\mathbb R^2\times T_{q}\mathbb R^2\rightarrow N_{p}M$, is the symmetric bilinear map
such that
$$\begin{array}{c}
II(\partial_{u},\partial_{u})=f_{uu}^{\perp}(q),\ II(\partial_{u},\partial_{v})=f_{uv}^{\perp}(q)\ \mbox{and}\
II(\partial_{v},\partial_{v})=f_{vv}^{\perp}(q).
\end{array}
$$

Given a vector $\nu\in N_{p}M$, the \emph{second fundamental form in
the direction $\nu$} of $M$ at $p$: $II_{\nu}:T_{q}\mathbb R^2\times T_{q}\mathbb R^2\rightarrow\mathbb{R}$
is defined as $II_{\nu}(X,Y)=\langle II(X,Y),\nu\rangle$, for all
$X,Y\in T_{q}\mathbb R^2$. The coefficients of $II_{\nu}$ in coordinates are
$$l_{\nu}(q)=\langle f_{uu}^{\perp},\nu\rangle(q),\ m_{\nu}(q)=\langle f_{uv}^{\perp},\nu\rangle(q)\ \mbox{and}\ n_{\nu}(q)=\langle
f_{vv}^{\perp},\nu\rangle(q).$$
For $X=a\partial_{u}+b\partial_{v}\in T_{q}\mathbb R^2$, we have
$II_{\nu}(X,X)=a^{2}l_{\nu}(q)+2abm_{\nu}(q)+b^{2}n_{\nu}(q)$ and fixing an
orthonormal frame $\{\nu_{1},\nu_{2}\}$ of $N_{p}M$,
$$\begin{array}{cl}\label{2ff}
II(X,X) & =II_{\nu_{1}}(X,X)+II_{\nu_{2}}(X,X)\\
        & =(a^{2}l_{\nu_{1}}+2abm_{\nu_{1}}+b^{2}n_{\nu_{1}})\nu_{1}+(a^{2}l_{\nu_{2}}+2abm_{\nu_{2}}+b^{2}n_{\nu_{2}})\nu_2,
\end{array}
$$
with the coefficients calculated in $q$. The
second fundamental form can also be represented by the matrix of coefficients
$$
\left(
  \begin{array}{ccc}
    l_{\nu_{1}} & m_{\nu_{1}} & n_{\nu_{1}} \\
    l_{\nu_{2}} & m_{\nu_{2}} & n_{\nu_{2}} \\
  \end{array}
\right).
$$
We identify $\R^2$ and $T_{q}\mathbb R^2$ by
$(x,y)\mapsto x\partial_u+y\partial_v$.
Let $C_q\subset T_{q}\mathbb R^2$ be the subset of
unit vectors:
$$
C_q=\{(x,y)\in T_{q}\mathbb R^2\,|\,x^2E(q)+2xyF(q)+y^2G(q)=1\},
$$
and let
$\eta_{q}:C_{q}\rightarrow N_{p}M$ be the map defined by
$$\eta_{q}(X)=II(X,X).$$
\begin{definition}\label{def:cp}
The image $\eta_{q}(C_{q})\subset N_pM$ is called
the \emph{curvature parabola} and is denoted by $\Delta_{p}$.
\end{definition}
The curvature parabola is a plane curve that may degenerate into a
line, a half-line or a point. Since $f$ has corank $1$ at
$q\in\mathbb R^2$, by changes of coordinates in the source and isometries in the target it can be written as
$f(u,v)=(u,f_{2}(u,v),f_{3}(u,v))$ with
$(f_{i})_{u}(q)=(f_{i})_{v}(q)=0$ for $i=2,3$.
Therefore $E=1$, $F=G=0$ and so $C_{q}=\{X=(\pm1,y)|\
y\in\mathbb{R}\}$. Fixing an orthonormal frame
$\{\nu_{1},\nu_{2}\}$ of $N_{p}M$,
\begin{equation}\label{parabola}
\eta(y)=(l_{\nu_{1}}+2m_{\nu_{1}}y+n_{\nu_{1}}y^{2})\nu_{1}+(l_{\nu_{2}}+2m_{\nu_{2}}y+n_{\nu_{2}}y^{2})\nu_{2}
\end{equation}
is a parametrisation for $\Delta_{p}$ in $N_{p}M$.

\begin{teo}[\cite{MartinsBallesteros}]\label{conditionsparabola}
Let $M\subset\mathbb{R}^{3}$ be a surface with a singularity of corank $1$ at $p\in M$. We
assume for simplicity that $p$ is the origin of $\mathbb{R}^{3}$ and denote by $j^{2}f(0)$
the $2$-jet of a local parametrisation $f:(\mathbb{R}^{2},0)\rightarrow(\mathbb{R}^{3},0)$
of $M$. Then the following holds:
\begin{itemize}
\item[(i)] $\Delta_{p}$ is a non-degenerate parabola if and only if $j^{2}f(0)\sim_{\mathcal{A}^{2}}(u,v^{2},uv)$;
\item[(ii)] $\Delta_{p}$ is a half-line if and only if $j^{2}f(0)\sim_{\mathcal{A}^{2}}(u,v^{2},0)$;
\item[(iii)] $\Delta_{p}$ is a line if and only if $j^{2}f(0)\sim_{\mathcal{A}^{2}}(u,uv,0)$;
\item[(iv)] $\Delta_{p}$ is a point if and only if $j^{2}f(0)\sim_{\mathcal{A}^{2}}(u,0,0)$.
\end{itemize}
Furthermore, if $f$ is given in Monge form such that
$$j^{2}f(0)=\left(u,\frac{1}{2}(a_{20}u^{2}+2a_{11}uv+a_{02}v^{2}),\frac{1}{2}(b_{20}u^{2}+2b_{11}uv+b_{02}v^{2})\right),$$
then the curvature parabola is parametrised by $$\eta(y)=(0,a_{20}+2a_{11}y+a_{02}y^{2},b_{20}+2b_{11}y+b_{02}y^{2})$$ and
\begin{itemize}
\item[(a)] $j^{2}f(0)\sim_{\mathcal{A}^{2}}(u,v^{2},uv)$ if and only if $a_{11}b_{02}-a_{02}b_{11}\neq0$;
\item[(b)] $j^{2}f(0)\sim_{\mathcal{A}^{2}}(u,v^{2},0)$ if and only if $a_{11}b_{02}-a_{02}b_{11}=0$ and $a_{02}^{2}+b_{02}^{2}>0$;
\item[(c)] $j^{2}f(0)\sim_{\mathcal{A}^{2}}(u,uv,0)$ if and only if $a_{02}=b_{02}=0$ and $a_{11}^{2}+b_{11}^{2}>0$;
\item[(d)] $j^{2}f(0)\sim_{\mathcal{A}^{2}}(u,0,0)$ if and only if $a_{02}=b_{02}=a_{11}=b_{11}=0$.
\end{itemize}
\end{teo}

A non zero tangent direction $X\in T_{q}\mathbb R^2$ is \emph{asymptotic}
if there is a non zero normal vector $\nu\in N_{p}M$ such that $II_{\nu}(X,Y)=0$,
for any $Y\in T_{q}\mathbb R^2$. Such a $\nu$ is called a
\emph{binormal direction}.

The parameter value $y\in\mathbb{R}$
corresponds to a unit tangent direction $X =\partial_{u}+y\partial_{v}\in C_{q}$. Denote by
$y_{\infty}$ the parameter value corresponding to the null tangent direction $X =\partial_{v}$.
If $\Delta_{p}$ degenerates to a line or a half-line, define
$\eta(y_{\infty})=\eta'(y_{\infty})=\eta'(y)/|\eta'(y)|$, where $y>0$ is any value such that
$\eta'(y)\neq0$. If $\Delta_{p}$ degenerates to a point $\nu$, then define
$\eta(y_{\infty})=\nu$ and $\eta'(y_{\infty})=0$. If $\Delta_{p}$ is a
non-degenerate parabola, $\eta(y_{\infty})$ and $\eta'(y_{\infty})$ are not defined.

\begin{lem}[\cite{MartinsBallesteros}]
A tangent direction in $T_{q}\mathbb R^2$ given by a parameter value $y\in\mathbb{R}\cup[y_{\infty}]$ is
asymptotic if and only if $\eta(y)$ and $\eta'(y)$ are collinear (provided they are defined).
\end{lem}

The parameter $y\in\mathbb{R}\cup[y_{\infty}]$ corresponding to an asymptotic direction
$X\in T_{q}\mathbb R^2$ is also called an asymptotic direction. The number of asymptotic directions is characterized by the topological type of the curvature parabola and when $\Delta_p$ is degenerate $y_{\infty}$ is an asymptotic direction (see also \cite{BenediniOset2} for an explanation):

\begin{itemize}
\item[(i)] If $\Delta_{p}$ is a non-degenerate parabola, there are $0,1$ or $2$ asymptotic
directions, according to the position of $p$: outside, on or outside the parabola,
respectively;
\item[(ii)] If $\Delta_{p}$ is a half-line such that the line that contains it does not pass through $p$, then there are two asymptotic directions,
$[y_{\nu},y_{\infty}]$, with $\eta(y_{\nu})$ being the vertex of $\Delta_{p}$, and if the line that contains it passes through $p$, then
every $y\in\mathbb{R}\cup[y_{\infty}]$ is an asymptotic direction;
\item[(iii)] If $\Delta_{p}$ is a line which does not pass through $p$ then $y_{\infty}$ is the only
asymptotic direction, and if it passes through $p$ then every $y\in\mathbb{R}\cup[y_{\infty}]$ is an asymptotic
direction;
\item[(iv)] If $\Delta_{p}$ is a point, every $y\in\mathbb{R}\cup[y_{\infty}]$
is an asymptotic direction.
\end{itemize}

\subsection{The umbilic curvature}

When $M$ is not a cross-cap singularity at $p$ the curvature parabola is degenerate. In this case a curvature can be defined.

We need to consider special frames on $N_pM$. When $\Delta_p$ is not a point (i.e. a half-line or a line), $y_{\infty}$ is well defined. Let $v_{\infty}$ be the binormal direction such that $\{\eta(y_{\infty}),v_{\infty}\}$ is an orthonormal positively oriented frame of $N_pM$. If $\Delta_p$ is a point which is not the origin then $\eta(y)$ is a non zero constant and we can consider the orthonormal frame given by $\{v,\eta(y)/|\eta(y)|\}$, where $v$ is a binormal direction. We call these frames \emph{adapted frames} of $N_pM$. When $\Delta_p$ is the origin, any frame is an adapted frame.

Given an adapted frame $\{\nu_1,\nu_2\}$ and $X\in C_q$ we have $$II(X,X)=II_{\nu_1}(X,X)\nu_1+II_{\nu_2}(X,X)\nu_2.$$ Notice that $II_{\nu_2}(X,X)$ does not depend on $X$ up to sign.

\begin{definition}[\cite{MartinsBallesteros}]
Given $X\in C_q$ and an adapted frame $\{\nu_1,\nu_2\}$ of $N_pM$ the \emph{umbilic curvature} of $M$ at $p$ is $$\kappa_u=|\langle II(X,X),\nu_2\rangle|=|II_{\nu_2}(X,X)|.$$
\end{definition}

Geometrically, $\kappa_u(p)$ measures the length of the projection of $\Delta_p$ on the infinity binormal direction when $\Delta_p$ is a line or a half-line and it is the distance between $\Delta_p$ and $p$ when $\Delta_p$ is a point.

\section{The axial curvature}

When $\Delta_p$ is a non-degenerate parabola, an adapted frame can be defined too. Let $v_d\in N_pM$ be the unitary vector in the direction of the directrix of the parabola and consider $v_a\in N_pM$ such that $\{v_a,v_d\}$ is a positively oriented orthonormal frame of $N_pM$. We call $v_a$ the \emph{axial vector} as it shares the direction of the axis of symmetry of the parabola when pointing towards the ``interior" of the parabola (see Figure \ref{frame}). When $\Delta_p$ is line or a half-line take $v_a=\eta(y_{\infty})$ and when $\Delta_p$ is a point which is not the origin $v_a$ is such that $\{v_a,\eta(y)/|\eta(y)|\}$ is an orthonormal positively oriented frame of $N_pM$.

\begin{figure}
\begin{center}
\includegraphics[width=0.5\linewidth]{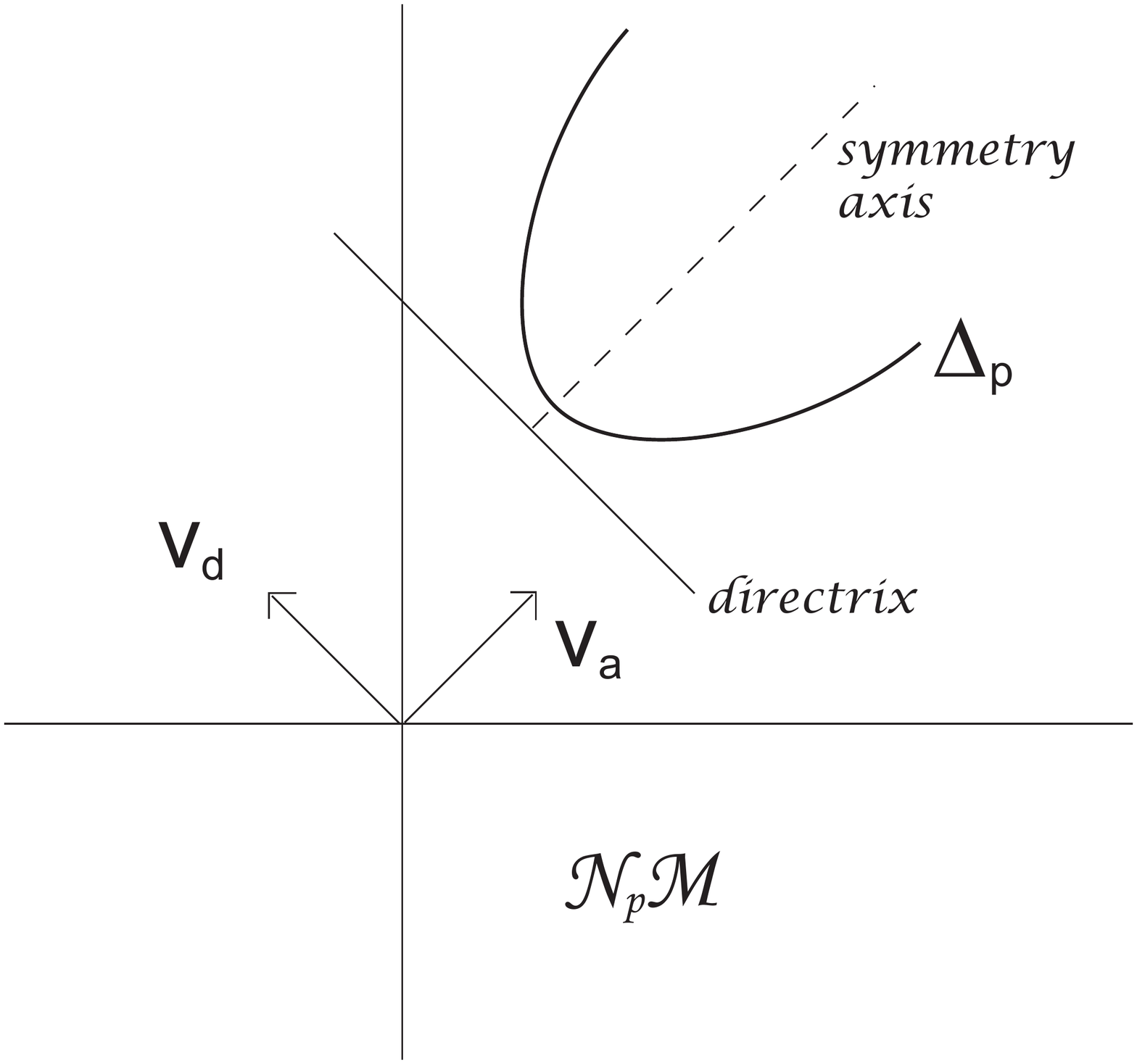}
\caption{Adapted frame for a non-degenerate parabola.}\label{frame}
\end{center}
\end{figure}

\begin{definition}
Given an adapted frame $\{v_a,\nu_2\}$ of $N_pM$ and $X\in T_q\mathbb R^2$ we define the \emph{axial normal curvature function} as $$K_{v_a}(X)=\langle II(X,X),v_a\rangle=II_{v_a}(X,X).$$ If $X\in C_q$, $X=\partial_u+y\partial_v$ and $II(X,X)=\eta(y)$ so we can consider $K_{v_a}(X)$ as a function on the parameter $y$. We call the number $$\kappa_a(p)=\min\{K_{v_a}(X):X\in C_q\}=\min\{\langle \eta(y),v_a\rangle:y\in\mathbb{R}\}$$ the \emph{axial curvature} of $M$ at $p$ (when it exists).
\end{definition}

Geometrically we have the following interpretations:

\begin{itemize}
\item[i)] When $\Delta_p$ is a non-degenerate parabola or a half-line, $\kappa_a(p)$ is the signed value of the extremal point of the projection of $\Delta_p$ on the direction given by $v_a$.
\item[ii)] When $\Delta_p$ is a line, the projection of $\Delta_p$ on the direction given by $v_a$ is the whole line and so $\kappa_a(p)$ is not bounded.
\item[iii)] When $\Delta_p$ is a point, the projection of $\Delta_p$ on the direction given by $v_a$ is the origin and so $\kappa_a(p)=0$.
\end{itemize}

\begin{rem}\label{principal}
In the case of regular surfaces in $\mathbb R^4$, the extremal points of the projection of the curvature ellipse in the direction orthogonal to a given normal direction $\nu$ are equal to the maximum and minimum values of the normal curvature in the direction $\nu$ and are therefore the $\nu$-principal curvatures (see Lemma 4.1 in \cite{nunoromerobringas}). In this sense $\kappa_a$ can be considered a ($v_a$-)principal curvature for $M$. This interpretation will be studied further in Section \ref{frontals}.
\end{rem}

From the definition and geometrical interpretation it follows that

\begin{prop}\label{zero}
$\kappa_a(p)=0$ if and only if $\Delta_p$ is a point, or $\eta(y_0)$ is parallel to $\nu_2$, where $y_0$ is a critical point of $K_{v_a}$ (see Figure \ref{ka0}).
\end{prop}

\begin{figure}
\begin{center}
\includegraphics[width=0.7\linewidth]{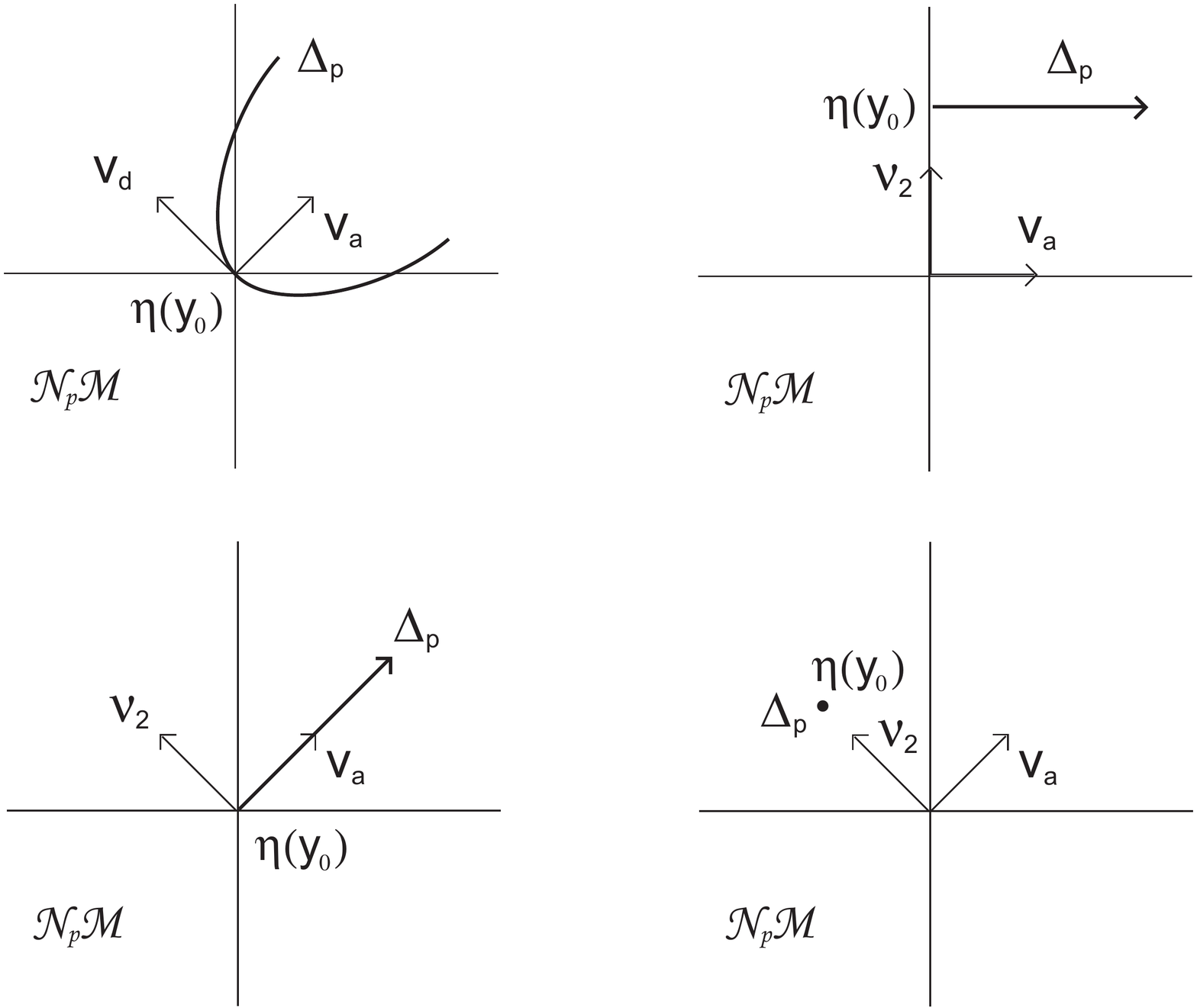}
\caption{Situations when the axial curvature is 0.}\label{ka0}
\end{center}
\end{figure}

For any $X\in T_q \mathbb R^2$, $II_{v_a}(X,X)=|X|^2II_{v_a}(\frac{X}{|X|},\frac{X}{|X|})$. So $$\kappa_a(p)=\min\{\frac{II_{v_a}(X,X)}{I(X,X)}:X\in T_q\mathbb R^2\}=min\{\frac{K_{v_a}(X)}{I(X,X)}:X\in T_q\mathbb R^2\}.$$

\begin{prop}\label{formula}
Let $M$ be given by the image of $f$ in Monge form such that
$j^{2}f(0)=\left(u,\frac{1}{2}(a_{20}u^{2}+2a_{11}uv+a_{02}v^{2}),\frac{1}{2}(b_{20}u^{2}+2b_{11}uv+b_{02}v^{2})\right)$
and such that $p\in M$ is the origin in $\mathbb R^3$. Suppose that $\Delta_p$ is a non-degenerate parabola or a half-line, then
\begin{equation}\label{eq:kappaamonge}
\kappa_a(p)=\frac{1}{\sqrt{a_{02}^2+b_{02}^2}}\big( (a_{20}a_{02}+b_{20}b_{02})-\frac{(a_{11}a_{02}+b_{11}b_{02})^2}{a_{02}^2+b_{02}^2}\big).
\end{equation}
\end{prop}
\begin{proof}
When $f$ is given as above, the curvature parabola is parameterised by $$\eta(y)=(0,a_{20}+2a_{11}y+a_{02}y^{2},b_{20}+2b_{11}y+b_{02}y^{2}).$$ If $\Delta_p$ is a non-degenerate parabola or a half-line, then $v_a=\frac{1}{\sqrt{a_{02}^2+b_{02}^2}}(a_{02},b_{02})$. (For completion we point out that in view of part 2 of Theorem \ref{conditionsparabola}, when $\Delta_p$ is a line we can take $v_a=\frac{1}{\sqrt{a_{11}^2+b_{11}^2}}(a_{11},b_{11})$ and when $\Delta_p$ is a point different from the origin we take $v_a=\frac{1}{\sqrt{a_{20}^2+b_{20}^2}}(-b_{20},a_{20})$.) So
\begin{align*}
K_{v_a}(y)&=\langle \eta(y),v_a\rangle\\
&=\frac{1}{\sqrt{a_{02}^2+b_{02}^2}}\big( (a_{20}a_{02}+b_{20}b_{02})+2(a_{11}a_{02}+b_{11}b_{02})y+(a_{02}^2+b_{02}^2)y^2 \big).
\end{align*}
We differentiate $K_{v_a}$ with respect to $y$ and equal to 0 to obtain the singular point $y_0=\frac{-(a_{11}a_{02}+b_{11}b_{02})}{a_{02}^2+b_{02}^2}$. Notice that $y_0$ is infinite if $\Delta_p$ is a line. Finally $$\kappa_a(p)=K_{v_a}(y_0)=\frac{1}{\sqrt{a_{02}^2+b_{02}^2}}\big( (a_{20}a_{02}+b_{20}b_{02})-\frac{(a_{11}a_{02}+b_{11}b_{02})^2}{a_{02}^2+b_{02}^2}\big).$$
\end{proof}

Furthermore, we have the following expression.
\begin{prop}\label{prop:general}
If $f$ satisfies that $\Delta_p$ is a non-degenerate parabola or a half-line,
and a coordinate system $(u,v)$ satisfies
$f_u(q)\ne0$ and $f_v(q)=0$,
then
\begin{align*}
\kappa_a
=&
\Big(
\inner{f_u}{f_u}
\big(\inner{f_u}{f_u}\inner{f_{vv}}{f_{vv}}
-\inner{f_u}{f_{vv}}^2\big)
\Big)^{-3/2}\\
&\hspace{2mm}
\Big(
\big(\inner{f_u}{f_{uu}} \inner{f_u}{f_{vv}}
- \inner{f_u}{f_u} \inner{f_{uu}}{f_{vv}}\big)
\big(\inner{f_u}{f_{vv}}^2
- \inner{f_u}{f_u} \inner{f_{vv}}{f_{vv}}\big)\\
&\hspace{5mm}
-\big(\inner{f_u}{f_{uv}} \inner{f_u}{f_{vv}}
- \inner{f_u}{f_u} \inner{f_{uv}}{f_{vv}}\big)^2
\Big)
\end{align*}
at $q$.
\end{prop}
To prove this proposition, we show the following lemma.
\begin{lem}\label{lem:special}
If $f$ satisfies that $\Delta_p$ is a non-degenerate parabola or a half-line,
and a coordinate system $(u,v)$ satisfies
$f_u(q)\ne0$, $f_v(q)=0$,
$|f_u(q)|=|f_{vv}(q)|=1$ and
$\inner{f_u(q)}{f_{vv}(q)}=0$,
then
\begin{equation}\label{eq:kaspecial}
\kappa_a(p)
=
\Big(\inner{f_{uu}}{f_{vv}}-\inner{f_{uv}}{f_{vv}}^2\Big)(q).
\end{equation}
\end{lem}
\begin{proof}
Firstly we show
\eqref{eq:kaspecial}
does not depend on the choice of the coordinate systems
satisfying
the assumption of the lemma.
Let $(x,y)=(x(u,v),y(u,v))$ be another coordinate system
satisfying that
$f_x(q)\ne0$, $f_y(q)=0$,
$|f_x(q)|=|f_{yy}(q)|=1$ and
$\inner{f_x(q)}{f_{yy}(q)}=0$.
Since
$$
f_u=f_xx_u+f_yy_u,\quad
f_v=f_xx_v+f_yy_v,
$$
and $f_y(q)=f_v(q)=0$, it holds that
$x_v(q)=0$.
Moreover, since $|f_x(q)|=1$ we have $x_u(q)^2=1$.
Furthermore, since
$$
f_{vv}(q)
=
f_x(q)x_{vv}(q)+f_{yy}(q)y_v(q)^2$$
and $|f_{yy}(q)|=|f_{vv}(q)|=1$,
$\inner{f_x(q)}{f_{yy}(q)}=\inner{f_u(q)}{f_{vv}(q)}=0$,
we have
$x_{vv}(q)=0$, $y_v(q)=1$.
Substituting
\begin{align*}
f_{uu}
=&
f_{xx}x_u^2+2f_{xy}x_uy_u+f_{yy}y_u^2+f_xx_{uu},\\
f_{uv}
=&
f_{xy}x_uy_v+f_{yy}y_uy_v+f_xx_{uv},
\end{align*}
into \eqref{eq:kaspecial},
we see
$$
\inner{f_{uu}}{f_{vv}}-\inner{f_{uv}}{f_{vv}}^2
=
\inner{f_{xx}}{f_{yy}}-\inner{f_{xy}}{f_{yy}}^2.
$$
Thus \eqref{eq:kaspecial} does not depend on the
coordinate system satisfying
the assumption of the lemma.
If $f(x,y)$ satisfies
\begin{align*}
j^{2}f(0)=&\Bigg(x,
\dfrac{a_{20}}{2}x^{2}+\dfrac{2a_{11}}{(a_{02}^2+b_{02}^2)^{1/4}}xy
+\dfrac{a_{02}}{2(a_{02}^2+b_{02}^2)^{1/2}}y^{2},\\
&\hspace{10mm}
\dfrac{b_{20}}{2}x^{2}+\dfrac{2b_{11}}{(a_{02}^2+b_{02}^2)^{1/4}}xy
+\dfrac{b_{02}}{2(a_{02}^2+b_{02}^2)^{1/2}}y^{2}\Bigg),
\end{align*}
then this satisfies the assumption of the lemma.
Under this coordinate system,
we easily see that
$$
\inner{f_{uu}}{f_{vv}}-\inner{f_{uv}}{f_{vv}}^2
$$
is equal to
\eqref{eq:kappaamonge}.
This shows the assertion.
\end{proof}
We remark that the existence of a coordinate system
of Lemma \ref{lem:special} can be shown easily.
\begin{proof}[Proof of Proposition \ref{prop:general}]
Let $(u,v)$ be a coordinate system satisfying
$$
f_u(q)\ne0,\
f_v(q)=0,\
|f_u(q)|=|f_{vv}(q)|=1\text{ and }
\inner{f_u(q)}{f_{vv}(q)}=0,
$$
and let
$(x,y)=(x(u,v),y(u,v))$ be another coordinate system satisfying
$f_y(q)=0$.
Then $x_v(q)=0$.
By
\begin{equation}
\label{eq:fuvdiff}
\begin{aligned}
f_u=&f_xx_u\\
f_{uu}=& y_u (f_{yy} y_u+x_u f_{xy})+x_u (y_u f_{xy}+x_u f_{xx})
+f_x x_{uu}\\
f_{uv}=&
y_v f_{yy} y_u+y_v x_u f_{xy}+f_x x_{uv}\\
f_{vv}=&
y_v^2 f_{yy}+x_{vv} f_x,
\end{aligned}
\end{equation}
at $q$, we have
\begin{equation}
\label{eq:xuyvxvv}
\begin{aligned}
x_u=&\pm1/\sqrt{\inner{f_x}{f_x}},\\
y_v=&
\left(\dfrac{\inner{f_x}{f_x}}{
\inner{f_x}{f_x}\inner{f_{yy}}{f_{yy}}
-\inner{f_x}{f_{yy}}^2}\right)^{1/4},\\
x_{vv}=&
-\left(
\dfrac{\inner{f_x}{f_x}}{
\inner{f_x}{f_x}\inner{f_{yy}}{f_{yy}}
-\inner{f_x}{f_{yy}}^2}\right)^{1/2}
\dfrac{\inner{f_x}{f_{yy}}}{\inner{f_x}{f_x}}
\end{aligned}
\end{equation}
Substituting \eqref{eq:fuvdiff} using
\eqref{eq:xuyvxvv} into \eqref{eq:kaspecial},
we have the assertion.
\end{proof}

\begin{ex}
\begin{itemize}
\item[i)] Consider the cross-cap singularity parameterised by $(u,u^2+v^2,2u^2+uv)$. The curvature parabola is a non-degenerate parabola and is parameterised by $\eta(y)=(2+2y^2,4+2y)$. In this case $v_a=(1,0)$ and $\kappa_a=2$.
\item[ii)] Consider the cuspidal edge parameterised by $f(u,v)=(u,\frac{a_{20}}{2}u^2+v^2,\frac{b_{20}}{2}u^2+v^3)$. The curvature parabola is a half-line parameterised by $\eta(y)=(a_{20}+2y^2,b_{20})$. Here $v_a=(1,0)$ and $\kappa_u=b_{20},\kappa_a=a_{20}.$
\end{itemize}
\end{ex}

\begin{rem}
Similarly to $\kappa_u$, $\kappa_a$ is independent of the choice of adapted frame of $N_pM$ and of parametrisation of $\Delta_p$ but may depend on the parametrisation of $M$. In fact, for the cuspidal edge parameterised by $f(u,v)=(u,u^2+v^2,v^3)$ the curvature parabola is parameterised by $\eta(y)=(2+2y^2,0)$ and $\kappa_u=0,\kappa_a=2$. On the other hand, for the same cuspidal edge parameterised by $f(u,v)=(u,u^2+(v^3+u)^2,(v^3+u)^3)$ the curvature parabola is the point $(4,0)$ and $\kappa_u=4,\kappa_a=0.$
\end{rem}

In \cite{hasegawahondaetal1}, it is proven that, taking a generic normal form for the cross-cap singularity $$f(u,v)=(u,c_{20}u^2+c_{11}uv+c_{02}v^2+O(3)(u,v),uv+O(3)(v)),$$ the coefficients $c_{20},c_{11},c_{02}$ are intrinsic invariants. Therefore, using Proposition \ref{formula} for this normal form, we get $\kappa_a(p)=2c_{20}-\frac{c_{11}^2}{2c_{02}}$, which means that the axial curvature is an intrinsic invariant for cross-cap singularities. On the other hand, we will prove in the next section that the axial curvature is equal to the singular curvature for frontals, which is also an intrinsic invariant (see \cite{sajietalannals}). Therefore, the axial curvature is an intrinsic invariant for frontals too. We can prove this in general.

\begin{prop}[Intrinsic formula for the axial curvature]
If $f$ satisfies that $\Delta_p$ is a non-degenerate parabola or a half-line, and $E,F$ and $G$ are the coefficients of the first fundamental form
$$
\kappa_a(p)
=
\frac
{\Big(
\big(\frac{E_u}{2} F_v
- E (F_{uv}-\frac{E_{vv}}{2})\big)
\big(F_v^2
- E \frac{G_{vv}}{2}\big)
-\big(\frac{E_v}{2} F_v
- E \frac{G_{uv}}{2}\big)^2
\Big)}{\Big(
E
\big(E\frac{G_{vv}}{2}
-F_v^2\big)
\Big)^{3/2}}
$$
where $E_u=\frac{\partial E}{\partial u}$ and so on.
\end{prop}
\begin{proof}
The formula follows by direct calculation and substitution in the formula of Proposition \ref{prop:general}. For instance, $E_{vv}=2\inner{f_u}{f_{uvv}}+2\inner{f_{uv}}{f_{uv}}$ and $F_{uv}=\inner{f_v}{f_{uuv}}+\inner{f_{uu}}{f_{vv}}+\inner{f_{uv}}{f_{uv}}+\inner{f_u}{f_{uvv}}$, so taking into account that in this coordinate system $f_v=0$ we get $\inner{f_{uu}}{f_{vv}}=F_{uv}-\frac{E_{vv}}{2}$.
\end{proof}

\section{The axial curvature for frontals: relation to the singular curvature}\label{frontals}

In this section we will show that the axial curvature is a generalization of the singular curvature for frontals.

A map-germ
$f:(\mathbb R^2,0)\rightarrow (\mathbb R^3,0)$
is a {\it frontal\/} if
there exists a well defined normal unit vector field $\nu$ along $f$,
namely, $|\nu|=1$ and for any $X\in T_q\R^2$, $df_q(X)\cdot\nu(q)=0$.
A frontal $f$ with a normal unit vector field $\nu$
is a {\it front} if the pair $(f,\nu)$ is an immersion.
Since at a cuspidal edge $f:(\mathbb R^2,0)\rightarrow (\mathbb R^3,0)$,
there is always a well defined normal unit vector field $\nu$ along $f$,
and the pair $(f,\nu)$ is an immersion, a cuspidal edge is a front.
On the other hand,
at a cuspidal cross-cap $f:(\mathbb R^2,0)\rightarrow (\mathbb R^3,0)$,
there is always a well defined normal unit vector field $\nu$ along $f$,
but the pair $(f,\nu)$ is not an immersion,
a cuspidal cross-cap is a frontal but not a front.
Let $f:(\mathbb R^2,0)\rightarrow (\mathbb R^3,0)$ be a frontal
with a normal unit vector field $\nu$.
Consider the function $\lambda=\det(f_x,f_y,\nu)$, where $(x,y)$ are the coordinates of $\mathbb R^2$.
Then $S(f)=\{\lambda^{-1}(0)\}$, where $S(f)$ is the set of singular
point of $f$.
A singular point $q$ is non-degenerate if $d\lambda(q)\neq 0$.
If $q$ is a non-degenerate singular point,
there is a well defined vector field $\eta$ in $\mathbb R^2$,
such that $df(\eta)=0$ on $S(f)$.
Such a vector field is called a {\it null vector field}.
A singular point $q$ is called of {\it first kind\/}
if $\eta\lambda(q)\ne0$.
A singular point $q$ is of first kind of a front if $f$
is a cuspidal edge (\cite{kokubuetal}).

Let $f:(\mathbb R^2,0)\rightarrow (\mathbb R^3,0)$ be a frontal
with a normal unit vector field $\nu$, and
$0$ a singular point of the first kind.
Since $\eta$ is transversal to $S(f)$, we can consider another vector field $\xi$ which is tangent to $S(f)$ and such that $(\xi,\eta)$ is positively oriented. Such a pair of vector fields is called an adapted pair. An adapted coordinate system $(u,v)$ of $\mathbb R^2$ is a coordinate system such that $S(f)$ is the $u$-axis, $\partial_v$ is the null vector field and there are no singular points besides the $u$-axis. Let $\gamma$ be a parametrisation of the singular curve $S(f)$ and let $\widehat\gamma=f\circ\gamma$. If $(u,v)$ is an adapted coordinate
system, then $f_{uv}=0$ holds on $S(f)$ and $\{f_u,f_{vv},\nu\}$ is linearly independent (in particular $f_{vv}(q)\neq 0$).

In \cite{martinssaji1} certain geometric invariants of cuspidal edges are studied. Amongst them are the singular curvature and the limiting normal curvature ($\kappa_s$ and $\kappa_{\nu}$), and these are given as follows:
\begin{equation}\label{eq:invdef1}
\kappa_s(t)=
\displaystyle
\operatorname{sgn}(d\lambda(\eta))\frac
{\det(\widehat\gamma'(t),\widehat\gamma''(t),\nu(\gamma(t)))}
{|\widehat\gamma'(t)|^3},\quad
\kappa_{\nu}(t)=
\displaystyle
\frac{\langle \widehat\gamma''(t),\nu(\gamma(t))\rangle}
{|\widehat\gamma'(t)|^2},
\end{equation}

A detailed description and geometrical interpretation of $\kappa_s$ and $\kappa_{\nu}$ can be found in \cite{sajietalannals}. In that paper, it is also shown that if $(u,v)$ is an adapted coordinate system, then $$\kappa_s(u,0)=sgn(\lambda_v)\frac{\det(f_u,f_{uu},\nu)}{|f_u|^3}(u,0).$$

There is a strong relation between the limiting normal curvature and the umbilic curvature, in fact, $\kappa_u$ is a generalization of $\kappa_{\nu}$ for non frontal singularities different from a cross-cap.

\begin{teo}[\cite{martinssaji1}]\label{kukn}
Let $f:(\mathbb R^2,q)\rightarrow (\mathbb R^3,p)$ be a map-germ, $q$ a cuspidal edge, and $\nu$ a
unit normal vector field along $f$. Then the following hold:
\begin{enumerate}
\item[i)] $\nu(q)$ is orthogonal to the line that contains $\Delta_p$ (i.e. $\nu=\nu_2$ of the adapted frame of $N_pM$).
\item[ii)] $\kappa_u(p)=|\kappa_{\nu}(q)|$
\item[iii)] $\kappa_s(q)=0$ if and only if $II(X,X)$ is parallel to $\nu$ at $p$, where $X$ is a non-zero tangent vector to $S(f)$ at $q$.
\item[iv)] $\kappa_s(q)=0=\kappa_u(p)$ if and only if $II(X,X)=0$ where $X$ is a non-zero tangent vector to $S(f)$ at $q$.
\end{enumerate}
\end{teo}

In the same spirit there is a strong relation between the axial curvature and the singular curvature:

\begin{teo}\label{axialsingular}
Let $f:(\mathbb R^2,q)\rightarrow (\mathbb R^3,p)$ be a map-germ, $q$ a non-degenerate frontal singularity, and $\nu$ a
unit normal vector field along $f$. Then the following hold:
\begin{itemize}
\item[i)] $\nu(q)$ is orthogonal to $v_a$
\item[ii)] $\kappa_a(p)=0$ if and only if $\Delta_p$ is a point, or $II(X,X)$ is parallel to $\nu$ at $p$, where $X$ is a non-zero tangent vector to $S(f)$ at $q$.
\item[iii)] $|\kappa_a(p)|=|\kappa_s(q)|$.
\end{itemize}
\end{teo}
\begin{proof}
From the definition of $v_a$, for the particular case of frontals, which have degenerate curvature parabola, $v_a=\nu_1$ where $\{\nu_1,\nu_2\}$ is an adapted frame of $N_pM$. From item i) in Theorem \ref{kukn}, $\nu(q)=\nu_2$, so $v_a$ is orthogonal to $\nu(q)$.

When $\Delta_p$ is a point, $\kappa_a=0$ by definition. If $\Delta_p$ is a line, then $\kappa_a$ is not bounded and item ii) does not apply. If $\Delta_p$ is a half-line then the minimum of $K_{v_a}$ is attained at the point where $\eta'(y)=0$. We consider an adapted coordinate system $(u,v)$. Recall that this implies that $f_{uv}=0$ holds on $S(f)$ and $f_{vv}(q)\neq 0$. The curvature parabola is the image of $C_q$ by the second fundamental form. We have $II(\partial u+y\partial v,\partial u+y\partial v)=f_{uu}^{\perp}(q)+2yf_{uv}^{\perp}(q)+y^2f_{vv}^{\perp}(q)=f_{uu}^{\perp}(q)+y^2f_{vv}^{\perp}(q)$, so $\eta'(y)=2yf_{vv}^{\perp}(q)$, which is 0 if and only if $y=0$. Therefore, the unitary tangent direction $X$ for which $II(X,X)$ is the extremal point of the half-line is $\partial_u$, which is tangent to $S(f)$ in the adapted coordinate system. On the other hand $\kappa_a(p)=\min\{K_{v_a}(X):X\in C_q\}=\langle \frac{II(\partial_u,\partial_u)}{I(\partial_u,\partial_u)},v_a\rangle=\langle \frac{II(\partial_u,\partial_u)}{E},v_a\rangle$, since $\partial u$ is the direction for which $K_{v_a}$ is minimum. Item ii) follows form $$\frac{1}{E}II(\partial_u,\partial_u)=\langle \frac{II(\partial_u,\partial_u)}{E},\nu\rangle\nu+\langle \frac{II(\partial_u,\partial_u)}{E},v_a\rangle v_a=\kappa_u\nu+\kappa_a v_a.$$

Now, for an adapted coordinate system we have $$\kappa_s(u,0)=sgn(\lambda_v)\frac{\det(f_u,f_{uu},\nu)}{|f_u|^3}(u,0)=sgn(\lambda_v)\langle \frac{f_{uu}^{\perp}}{|f_u|^2},\nu\times\frac{f_u}{|f_u|}\rangle.$$ From i) $\nu\times\frac{f_u}{|f_u|}=v_a$ so the above equation is equal to $$sgn(\lambda_v)\langle \frac{f_{uu}^{\perp}}{E},v_a\rangle=sgn(\lambda_v)\langle \frac{II(\partial_u,\partial_u)}{I(\partial_u,\partial_u)},v_a\rangle=sgn(\lambda_v)\kappa_a(p).$$

\end{proof}

\begin{rem}
An alternative way to prove iii) in Theorem \ref{axialsingular} is using the formula \eqref{eq:kaspecial} since when $f$ is a frontal,
$$
\nu=f_u\times f_{vv}/|f_u\times f_{vv}|.
$$
So by direct calculation, $\kappa_a$ is equal to $\kappa_s$ in the case non-degenerate frontals.
\end{rem}

By Corollary 1.14 of \cite{sajietalannals}, for non-degenerate front singularities of the second kind (swallowtail), the singular curvature is unbounded. In fact, this is a corollary of the above Theorem too since the $2$-jet of such a singularity is $\mathscr A$-equivalent to $(x,xy,0)$ and the curvature parabola is a line, so the minimum of the projection to $v_a$ (i.e. $\kappa_a$) is unbounded.

\begin{rem}
In \cite{teramoto} the author defines some principal curvatures for wave fronts and in Theorem 3.1 he proves that when the singularity is of first or second kind then one of these principal curvatures is bounded and in fact is equal to $\kappa_{\nu}$. By Theorem \ref{kukn} $\kappa_u=\kappa_{\nu}$ and $\nu=\nu_2$ of the adapted frame of $N_pM$, so $\kappa_u$ can be seen as a $\nu_2$-principal curvature for $M$. $\kappa_u$ is the projection on the direction given by $\nu_2$ and $\kappa_a$ is the extremal point of the projection on the direction given by $v_a$, so it makes sense to consider $\kappa_a$ as a kind of $v_a$-principal curvature for wave fronts, which is consistent with the interpretation given in Remark \ref{principal}.
\end{rem}

\begin{rem}
If $f$ is given in Monge form and $\Delta_p$ is not a line or a non-degenerate parabola (when $\Delta_p$ is a line, $\kappa_a$ is unbounded and when $\Delta_p$ is a non-degenerate parabola, $\kappa_u$ is not defined), then $a_{11}^2+b_{11}^2=0$ and $\kappa_a^2+\kappa_u^2=a_{20}^2+b_{20}^2$. This corresponds to the curvature of the curve $\gamma(t)=f(t,0)$. For the case of frontals in an adapted coordinate system this curve is the cuspidal edge and its curvature $\kappa$ satisfies $\kappa^2=\kappa_a^2+\kappa_u^2$ (see \cite{martinssaji1}).
\end{rem}

\section{Contact with planes}

In this section we consider the contact of $M$ with the plane orthogonal to $v_a$, which we denote by $v_a^{\perp}$. The contact of $M$ with a plane orthogonal to a vector $v$ is measured by the singularities of the height function in the direction $v$ $$h_v:M\rightarrow\mathbb{R},\ h_v(p)=\langle p,v\rangle.$$ We study the height function in the direction $v_a$ and obtain geometric interpretations for the axial curvature.

We first show the following lemma.
\begin{lem}
If $\Delta_p$ is a non-degenerate parabola or a half-line
$($i.e. $j^2f$ is equivalent to $(x,y^2,xy)$ or $(x,y^2,0))$,
then for a coordinate system as in Proposition \ref{prop:general},
$v_a=\frac{(f_u\times f_{vv})\times f_u}{|(f_u\times f_{vv})\times f_u|}(q)$, and in the coordinate system of Lemma \ref{lem:special}, $v_a=f_{vv}(q)$.
\end{lem}
\begin{proof}
Let $(u,v)$ be a coordinate system as in Proposition \ref{prop:general}.
Since $j^2f$ is equivalent to $(x,y^2,xy)$ or $(x,y^2,0)$, $f_u\times f_{vv}\ne0$ at $q$,
and
$
N_pM=\langle
V_1,\ V_2
\rangle_{\R}
$,
where
$$
V_1=\dfrac{f_u(q)\times f_{vv}(q)}{|f_u(q)\times f_{vv}(q)|},\quad
V_2=\dfrac{(f_u(q)\times f_{vv}(q))\times f_{u}(q)}
{|(f_u(q)\times f_{vv}(q))\times f_{u}(q)|}.
$$
Since $f_v(q)=0$, the condition
$I(a\partial_u+b\partial_v,a\partial_u+b\partial_v)=1$
is equivalent to $a=1/\inner{f_u(q)}{f_u(q)}^{1/2}$.
Thus the curvature parabola is
\begin{align*}
&\Bigg\{
\bigg(\dfrac{\inner{f_{uu}(q)}{V_1}}{E(q)}
+\dfrac{2b\inner{f_{uv}(q)}{V_1}}{E(q)^{1/2}}
+b^2\inner{f_{vv}(q)}{V_1}\bigg)V_1\\
&\hspace{20mm}+
\bigg(\dfrac{\inner{f_{uu}(q)}{V_2}}{E(q)}
+\dfrac{2b\inner{f_{uv}(q)}{V_2}}{E(q)^{1/2}}
+b^2\inner{f_{vv}(q)}{V_2}\bigg)V_2\Bigg|
b\in\R\Bigg\}.
\end{align*}
Since $\inner{f_{vv}(q)}{V_1}=0$,
there is no $b^2$-term in the coefficient of $V_1$.
Therefore
the curvature parabola is a parabola whose axis is
parallel to the direction of $V_2$, and
we obtain
$$
v_a=V_2=
\dfrac{(f_u(q)\times f_{vv}(q))\times f_{u}(q)}
{|(f_u(q)\times f_{vv}(q))\times f_{u}(q)|}.
$$
Thus the second assertion is shown.
The first assertion immediately follows from the second assertion.
\end{proof}

\begin{prop}\label{a1}
If $f$ satisfies that $\Delta_p$ is a non-degenerate parabola or a half-line, the singularities of $h_{v_a}$, the height function in the direction $v_a$, are
\begin{itemize}
\item[i)] $A_1^+$ if and only if $\kappa_a(p)>0$,
\item[ii)] $A_1^-$ if and only if $\kappa_a(p)<0$,
\item[iii)] $A_{\geq 2}$ if and only if $\kappa_a(p)=0$. In particular, $A_2$ if and only if $\kappa_a=0$ and
\begin{align*}
&\Big(-\inner{f_{uuu}}{f_{vv}}
+3\inner{f_{uuv}}{f_{vv}}\inner{f_{uv}}{f_{vv}}\\
&\hspace{10mm}
-3\inner{f_{uvv}}{f_{vv}}\inner{f_{uv}}{f_{vv}}^2
+\inner{f_{vvv}}{f_{vv}}\inner{f_{uv}}{f_{vv}}^3\Big)(q)\neq0
\end{align*}
\end{itemize}
\end{prop}
\begin{proof}
If a coordinate system $(u,v)$ satisfies
$f_u(q)\ne0$, $f_v(q)=0$,
$|f_u(q)|=|f_{vv}(q)|=1$ and
$\inner{f_u(q)}{f_{vv}(q)}=0$,
then, by Lemma \ref{lem:special}, $$\kappa_a(p)=
\Big(\inner{f_{uu}}{f_{vv}}-\inner{f_{uv}}{f_{vv}}^2\Big)(q).$$
Since $\inner{f_{vv}}{f_{vv}}(q)=1$,
the Hessian matrix of $h_{v_a}$ at $q$ is
$$
{\mathcal H}=
\begin{pmatrix}
\inner{f_{uu}}{f_{vv}}&\inner{f_{uv}}{f_{vv}}\\
\inner{f_{uv}}{f_{vv}}&1
\end{pmatrix}(q).
$$
Notice that, $\det{\mathcal H}$ is precisely $\kappa_a(p)$.
Thus, assertions i) and ii) are shown.

We assume $\det {\mathcal H}=0$. Since $\inner{f_{vv}}{f_{vv}}(q)\ne0$,
${\rm rank} {\mathcal H}=1$.
We set $k_H=(-1,\inner{f_{uv}}{f_{vv}})$.
Then $k_H$ spans the kernel of ${\mathcal H}$.
It is known that $h_{v_a}$ is an $A_2$-singularity at $p$ if and only if
$$
\bigg(
\Big(-\partial_u
+\inner{f_{uv}}{f_{vv}}(q)\partial_v\Big)\det {\mathcal H}\bigg)(q)\ne0
$$
By a straightforward calculation, it is equivalent to
\begin{align*}
&-\inner{f_{uuu}}{f_{vv}}
+3\inner{f_{uuv}}{f_{vv}}\inner{f_{uv}}{f_{vv}}
-\inner{f_{uvv}}{f_{vv}}\inner{f_{uu}}{f_{vv}}\\
&\hspace{30mm}
-2\inner{f_{uvv}}{f_{vv}}\inner{f_{uv}}{f_{vv}}^2
+\inner{f_{vvv}}{f_{vv}}\inner{f_{uv}}{f_{vv}}\inner{f_{uu}}{f_{vv}}\neq0
\end{align*}
at $q$.
By the assumption $\kappa_a=\inner{f_{uu}}{f_{vv}}-\inner{f_{uv}}{f_{vv}}^2=0$, so
we get assertion iii).
\end{proof}

\begin{coro}\label{side}
If $f$ satisfies that $\Delta_p$ is a non-degenerate parabola or a half-line, then the surface $M$ is (locally) only on one side of the osculating plane if and only if $\kappa_a>0$
\end{coro}

\begin{ex}
Given a cuspidal edge $f(u,v)=(u,\frac{a_{20}}{2}u^2+\frac{v}{2},v^3)$, then $\kappa_a=\kappa_s=a_{20}$. When $a_{20}>0$ (resp. $a_{20}<0$) the cuspidal edge is positively curved (resp. negatively curved). See Figure \ref{cuspidaledges}.
\begin{figure}
\begin{center}
\includegraphics[width=0.6\linewidth]{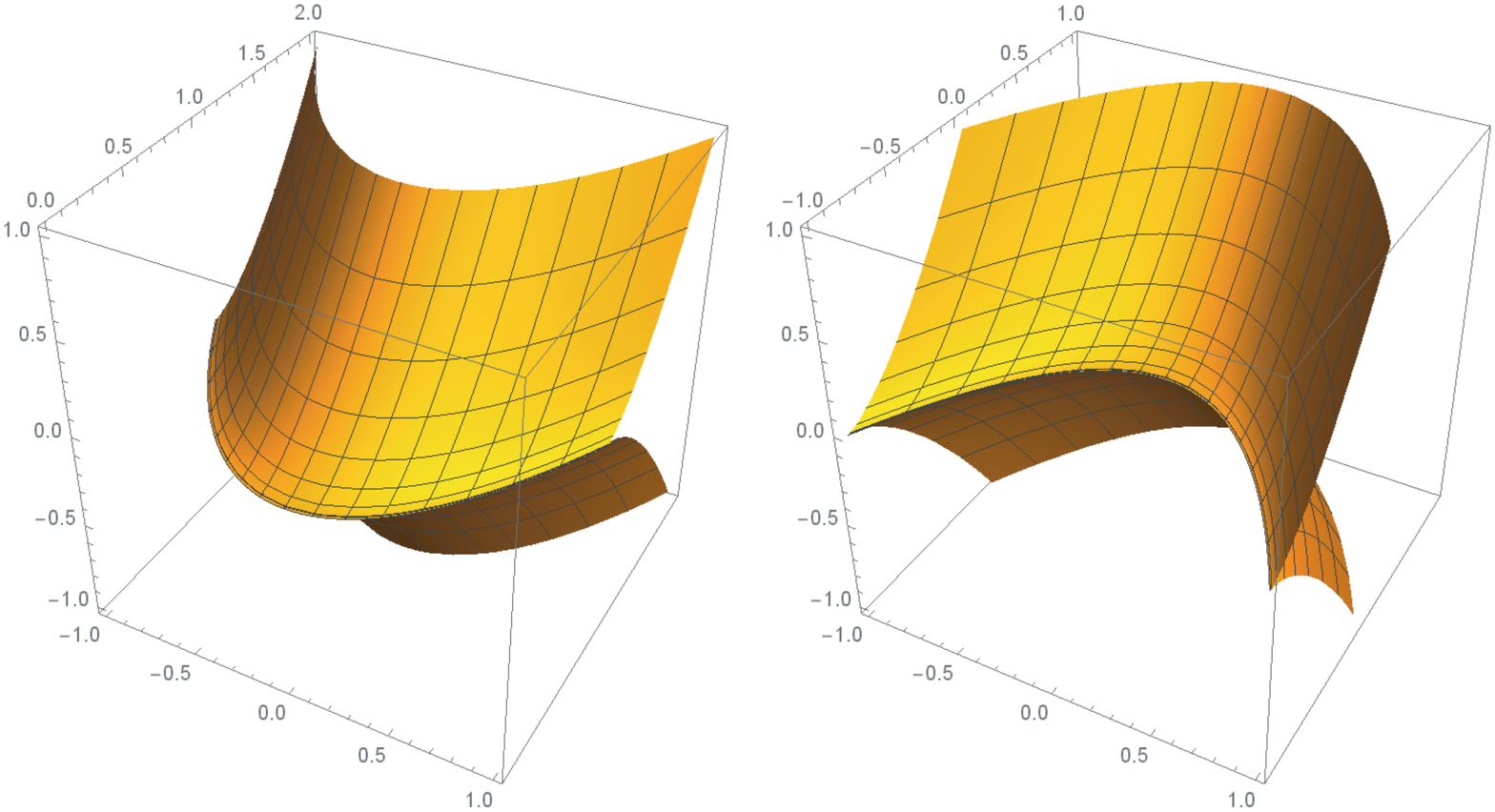}
\caption{The left one is positively curved and the right one is negatively curved.}\label{cuspidaledges}
\end{center}
\end{figure}
\end{ex}

\begin{prop}\label{binormal}
Suppose $\kappa_a(p)$ is defined (i.e. $\Delta_p$ is not a line), then $\kappa_a(p)=0$ if and only if $v_a$ is a binormal direction.
\end{prop}
\begin{proof}
By Proposition \ref{zero}, $\kappa_a(p)=0$ if and only if $\Delta_p$ is a point, or $\eta(y_0)$ is parallel to $\nu_2$, where $y_0$ is a critical point of $K_{v_a}$. When $\Delta_p$ is a point, all directions are asymptotic and by definition $v_a$ is orthogonal to $\eta(y)$ for any $y$, thus, $v_a$ is a binormal direction. When $\Delta_p$ is not a point, $\{v_a,\nu_2\}$ is an adapted frame. If $\eta(y_0)$ is parallel to $\nu_2$, and $y_0$ is a critical point of $K_{v_a}$, this means that $y_0$ is an asymptotic direction, and since $\eta(y_0)$ and $v_a$ are orthogonal, $v_a$ is a binormal direction.
\end{proof}

With this we can recover part of Theorem 2.15 in \cite{MartinsBallesteros} and give some more information:

\begin{prop}
$h_{v_a}$ has a degenerate singularity if and only if $v_a$ is a binormal direction. Moreover, the singularity is of corank 2 if and only if $\Delta_p$ is degenerate and $\kappa_a(p)=\kappa_u(p)=0$.
\end{prop}
\begin{proof}
The first assertion follows directly from Propositions \ref{a1} and \ref{binormal} and is also found in Theorem 2.15 in \cite{MartinsBallesteros}. On the other hand, their result states that the singularity of the height function in a direction $v$ is of corank 2 if and only if $\Delta_p$ is degenerate and $\kappa_u(p)=0$ and $v$ is an infinite binormal direction. This, together with Proposition \ref{binormal} give that the singularity of $h_{v_a}$ is of corank 2 if and only if $\Delta_p$ is degenerate and $\kappa_a(p)=\kappa_u(p)=0$.
\end{proof}

Given a surface $M\subset\mathbb{R}^{3}$ with corank $1$ singularity at $p\in M$, the point $p$ is called elliptic, hyperbolic, parabolic or inflection according to whether there are 0, 2, 1 or infinite asymptotic directions at that point (see \cite{benedinioset}). Equivalently in \cite{OsetSinhaTari}, the point is elliptic, hyperbolic or parabolic according to whether the $GL(2,{\mathbb R})\times GL(2,{\mathbb R})$-orbit of the pair $(j^{2}f_{2}(u,v),j^{2}f_{3}(u,v))$ is of elliptic, hyperbolic or parabolic type. These two definitions coincide. Differently from the regular case, for a singular point, being elliptic or hyperbolic does not ensure the existence of an osculating plane such that the surface is locally on one side of the plane. This is distinguished by the sign of $\kappa_a$. In fact, the sign of $\kappa_a$ does not always imply whether the point is elliptic or hyperbolic, however it does imply the ``ellipticity" or ``hyperbolicity" in the ``regular sense", that is, whether the surface is only on one side of the osculating plane or on both. If $\Delta_p$ is a non-degenerate parabola (i.e. $p$ is a cross-cap ) and $\kappa_a>0$ then $p$ is a hyperbolic point (since there are two asymptotic directions), but if $\kappa_a<0$, then it can be hyperbolic, elliptic or parabolic. If $\kappa_a=0$, the point can be hyperbolic or parabolic.

\begin{ex}
In \cite{FukuiHasegawa} and \cite{West} it is proven that with changes of coordinates in the source and isometries in the target a cross-cap can be parametrised by $f(u,v)=(u,c_{20}u^2+c_{11}uv+c_{02}v^2+O(3)(u,v),uv+O(3)(v)),$ with $c_{02}\neq0$. The cross-cap is called hyperbolic, elliptic or parabolic if $c_{20}$ is negative, positive or zero (\cite{West}). A cross-cap is hyperbolic, elliptic or parabolic if and only if the singular point is elliptic, hyperbolic or parabolic in the above sense (\cite{BallesterosTari},\cite{OsetSinhaTari}). Consider the case $f(u,v)=(u,u^2-3uv+v^2,uv)$ which is an elliptic cross-cap (hyperbolic point) and has two asymptotic directions. Here $\kappa_a=-\frac{5}{2}$ is negative and so, by Corollary \ref{side} the surface is on both sides of $v_a^{\perp}$. On the other hand, consider $f(u,v)=(u,-u^2+v^2,uv)$, which is a hyperbolic cross-cap (elliptic point) and has no asymptotic directions. Here $\kappa_a=-2$ is also negative and so the surface is also on both sides of $v_a^{\perp}$. See Figure \ref{crosscaps}.
\begin{figure}
\begin{center}
\includegraphics[width=0.7\linewidth]{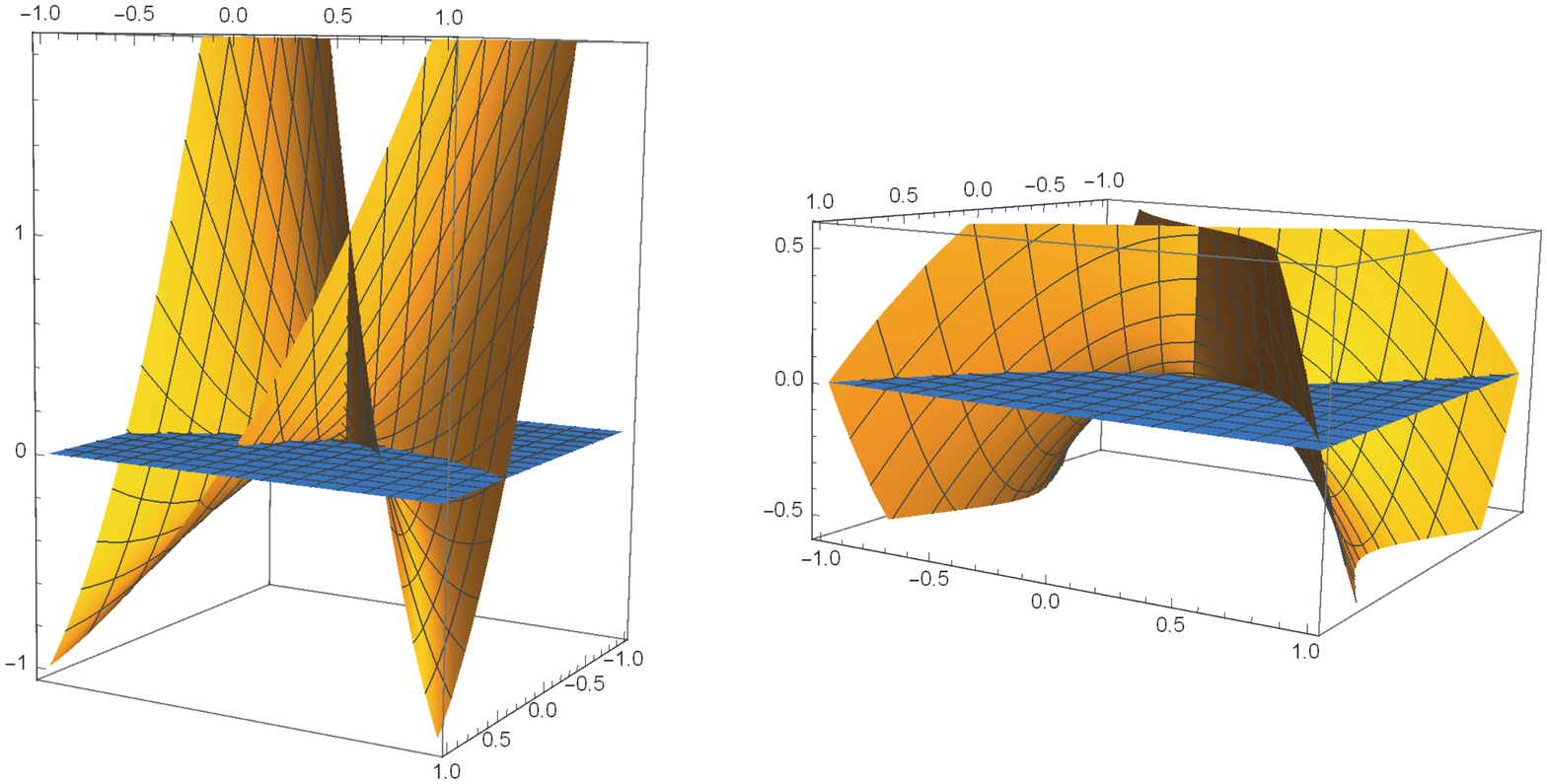}
\caption{Both the elliptic cross-cap (left) and the hyperbolic cross-cap (right) lie on both sides of $v_a^{\perp}$.}\label{crosscaps}
\end{center}
\end{figure}
\end{ex}

In order to distinguish when a cross-cap with negative axial curvature is an elliptic, hyperbolic or parabolic cross-cap we have the following criteria.

\begin{prop}\label{crosscaptype}
Let $f$ be such that $\Delta_p$ is a non-degenerate parabola and suppose $\kappa_a<0$. Then $p$ is an elliptic cross-cap (resp. hyperbolic cross-cap) if and only if the intersection of $M$ with the plane $v_a^{\perp}$ is two tangent quadratic curves which lie in the same half-plane (resp. in different half-planes). Moreover, $p$ is a parabolic cross-cap if and only if one of the curves is a straight line.
\end{prop}
\begin{proof}
Consider the parametrisation of type $f(u,v)=(u,c_{20}u^2+c_{11}uv+c_{02}v^2+O(3)(u,v),uv+O(3)(v)),$ then the coordinate system of Lemma \ref{lem:special} is satisfied. The intersection of $v_a^{\perp}$ with $M$ is given by $\inner{f_{vv}}{f}=0$. This is equal to the height function in the direction $v_a$, $h_{v_a}$, which we denote by $h$. The Hessian of the height function is equal to $\kappa_a$ and there is an $A_1^-$ singularity when it is negative. The two solutions for the intersection of $M$ with the osculating plane are given by $$\frac{-h_{uv}\pm\sqrt{-\kappa_a}}{h_{uu}}=\frac{-h_{uv}\pm\sqrt{h_{uv}^2-h_{uu}h_{vv}}}{h_{uu}}.$$ We denote these solutions by $a_1$ and $a_2$. So the zero level curves of the height function in the source ($h^{-1}(0)$) are parameterised by $(a_it+\ldots,t)$ for $i=1,2$, where $\ldots$ represents higher order terms. The image of these two curves is $f(a_it+\ldots,t)=(a_it+\ldots,\ldots,a_it^2+\ldots)$, so we get two tangent quadratic curves in the osculating plane if $a_i\neq 0$ for $i=1,2$. If one of the solutions is zero (i.e. $h_{uu}h_{vv}=0$), then one of the curves is parameterised by $(t,0)$ and the image is a straight line $(t,0,0)$.

On the other hand the solutions $a_1$ and $a_2$ have the same sign (resp. opposite) if $h_{uu}h_{vv}>0$ (resp. $h_{uu}h_{vv}<0$). Using $\inner{f_{vv}}{f_{vv}}=1$, we get $h_{uu}h_{vv}=\inner{f_{uu}}{f_{vv}}$ and $\inner{f_{uu}}{f_{vv}}>0$ (resp. $\inner{f_{uu}}{f_{vv}}<0$) if and only if $c_{20}>0$ (resp. $c_{20}<0$). This means that the curves lie in the same half-plane if and only if the cross-cap is elliptic and in different half-planes if and only if the cross-cap is hyperbolic. Moreover, one of the solution is 0 if and only if $h_{uu}h_{vv}=\inner{f_{uu}}{f_{vv}}=c_{20}=0$ (which means that the cross-cap is parabolic).
\end{proof}

\begin{ex}
\begin{itemize}
\item[i)] Considering the cross-caps of the above example, the intersection of $M$ with $v_a^{\perp}$ for the elliptic cross-cap $f(u,v)=(u,u^2-3uv+v^2,uv)$ is given by the curves $(\frac{3\pm\sqrt{5}}{2}t,0,\frac{3\pm\sqrt{5}}{2}t^2)$, both of which are in the same half-plane. On the other hand, for the hyperbolic cross-cap $f(u,v)=(u,-u^2+v^2,uv)$ the curves are given by $(t,0,\pm t^2)$, which are in different half-planes. See Figure \ref{osculating}.
\begin{figure}
\begin{center}
\includegraphics[width=0.7\linewidth]{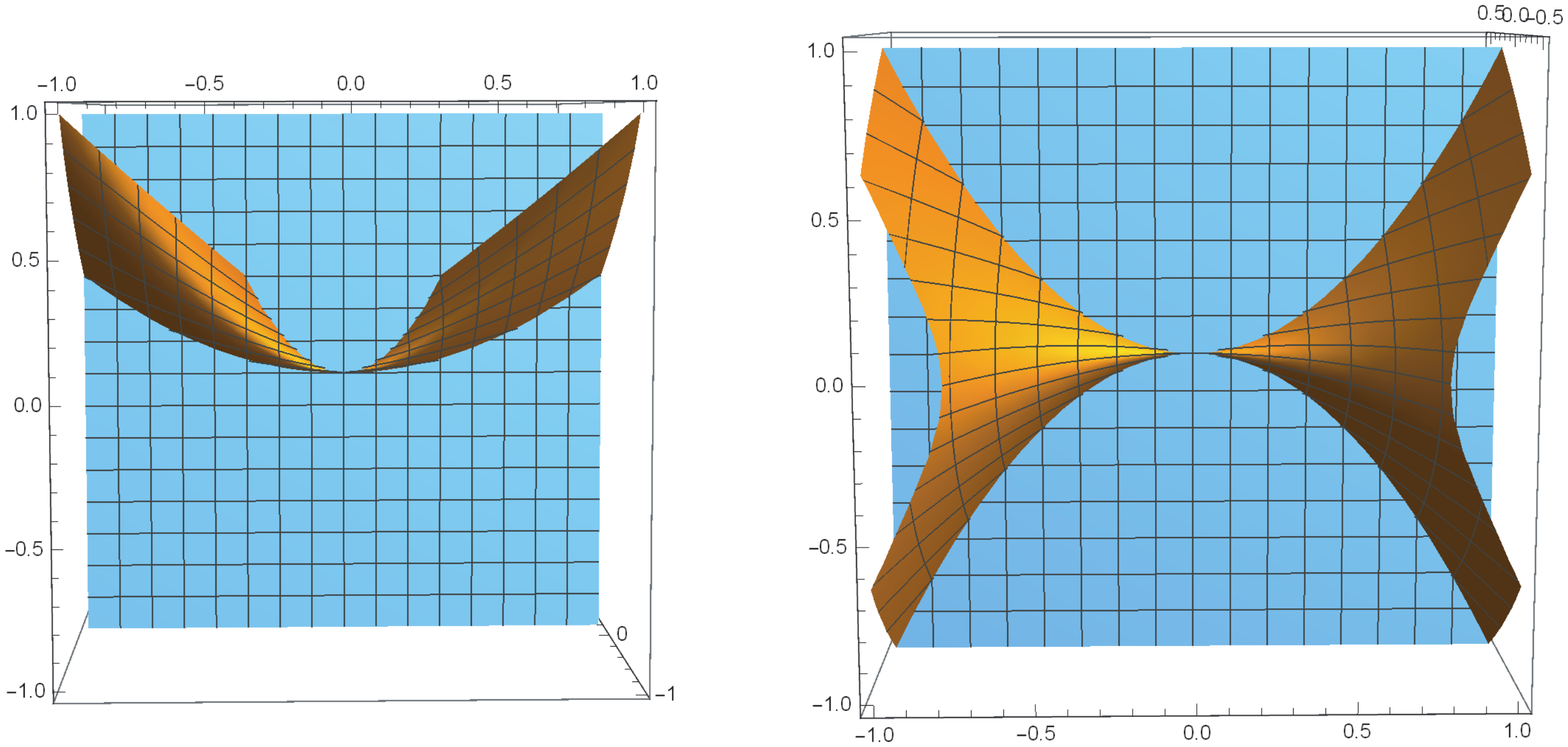}
\caption{The left one is the elliptic cross-cap and the right one the hyperbolic one.}\label{osculating}
\end{center}
\end{figure}
\item[ii)] Consider the parabolic cross-cap given by $f(u,v)=(u,-3uv+v^2,uv)$. Here $\kappa_a=\frac{-9}{2}<0$. The intersection of $M$ with $v_a^{\perp}$ is given by the curves $(t,0,3t^2)$ and $(t,0,0)$.
\end{itemize}
\end{ex}

Similarly to above, when $\Delta_p$ is a half-line and $\kappa_a<0$, then the singular point $p$ can be hyperbolic (if the line that contains $\Delta_p$ does not pass through $p$, i.e. it has two asymptotic directions) or an inflection point (if the line that contains $\Delta_p$ passes through $p$, i.e. it has infinite asymptotic directions). There are criteria to distinguish these two situations, however these criteria depend on the type of singularity. For example, for cuspidal edges we have the following

\begin{prop}
Let $f$ be cuspidal edge and suppose $\kappa_a<0$. Then $p$ is a hyperbolic point if and only if the intersection of $M$ with $v_a^{\perp}$ is two tangent cubic curves which meet at two local maxima or two local minima (i.e. they lie in the same half-plane). Moreover, $p$ is an inflection point if and only if at least one of the curves has an inflection point at the origin of $v_a^{\perp}$.
\end{prop}
\begin{proof}
The proof is similar to that of Proposition \ref{crosscaptype} but in some ways gives more information.

If $f$ is a cuspidal edge then $\Delta_p$ is a half-line. Then there exists a coordinate system such that
$|f_u(q)|=|f_{vv}(q)|=1$, $f_v(q)=0$, $\inner{f_u}{f_{vv}}(q)=0$.
Notice that $v_a=f_{vv}(q)$.

Then on this coordinate system,
$\det(f_u,f_{uv},f_{vv})(q)=0$. This condition together with $|f_{vv}(q)|=1$ is equivalent to item b) in Theorem \ref{conditionsparabola}.

We define $h_{v_a}=h=\inner{f}{v_a}$.
Then $h_u(p)=h_v(p)=0$, and
by the assumption $\kappa_a<0$, it holds that
$\det {\rm Hess} h(p)<0$.
Thus $h^{-1}(0)$ is two transversal curves.
Moreover, if $\inner{f_{vv}}{f_{vv}}(q)\ne0$,
these curves are not tangent to the $v$-axis.
Thus these curves can be parametrized as
$$
(u,c(u)).
$$
This $c$ satisfies
$h_u+h_vc'=0$, and
$h_{uu}+2h_{uv}c'+h_{vv}(c')^2=0$.
Since $h_{vv}(p)\ne0$,
\begin{equation}\label{eq:cprime}
c'(0)
=
\dfrac{-h_{uv}\pm\sqrt{h_{uv}^2-h_{uu}h_{vv}}}{h_{vv}}.
\end{equation}
We denote by
$c^+(u)$ the case of $c(u)$
which satisfies $(c^+)'(0)$ is equal to \eqref{eq:cprime} with the ``$+$'' sign, and
by $c^-(u)$ the case when $(c^-)'(0)$ is equal to \eqref{eq:cprime} with the ``$-$'' sign.
The intersection curves of $v_a^\perp$ and $f$
are
$\widetilde c^\pm(u)=f(u,c^\pm(u))$.
We have $(\widetilde c^\pm)'(0)=f_u(0,0)(\ne0)$, therefore these curves lie on one half plane
if and only if
the signs of $a^+(u)$ and $a^-(u)$ are the same for $u$ small, where
$$
a^\pm(u)=
\det(f_u(q), (\widetilde c^\pm)(u),f_{vv}(q)).
$$

We see that $(a^\pm)'(0)=0$.
On the other hand, since
$$
(\widetilde c^\pm)''(u)
=
f_{uu}(q)+2f_{uv}(q)c'(0)+f_{vv}(q)(c'(0))^2,
$$
and
$\det(f_u,f_{uv},f_{vv})(q)=0$,
$$
(a^\pm)''(0)=
\det(f_u(q),f_{uu}(q),f_{vv}(q)).
$$
If $\det(f_u(q),f_{uu}(q),f_{vv}(q))\ne0$,
then the sign of $(a^\pm)''(0)$ is equal to the sign of
$\det(f_u(q),f_{uu}(q),f_{vv}(q))$,
namely, $(a^\pm)''(0)$
does not depend on the sign $\pm$.
Thus both $\widetilde c^\pm$ lie on one half-plane. In fact, they meet tangentially at local minima or maxima.

The condition $\det(f_u(q),f_{uu}(q),f_{vv}(q))=0$ means that $f_{uu}$ and $f_{vv}$ are parallel. On the other hand, $\det(f_u(q),f_{uv}(q),f_{vv}(q))=0$ means that $f_{uv}$ is parallel to $f_{vv}$ too. So $II(\partial u+y\partial v,\partial u+y\partial v)=f_{uu}^{\perp}(q)+2yf_{uv}^{\perp}(q)+y^2f_{vv}^{\perp}(q)=\phi(y)f_{vv}(q)$, which means that the line containing $\Delta_p$ passes through $p$. This means that $p$ is an inflection point. Similarly $\det(f_u(q),f_{uu}(q),f_{vv}(q))\ne0$ means that $p$ is a hyperbolic point.

If $\det(f_u(q),f_{uu}(q),f_{vv}(q))=0$,
we have to look at $(a^\pm)'''(0)$.
\begin{align}
(a^\pm)'''(0)=&
\det\Big(f_u(q),
f_{uuu}(q)+3f_{uuv}(q)c'(0)\nonumber\\
&\hspace{20mm}+3f_{uvv}(q)(c'(0))^2+f_{vvv}(q)(c'(0))^3,f_{vv}(q)\Big).
\label{eq:apmppp}
\end{align}
For the case of cuspidal edges we take the normal form given in \cite{martinssaji1}, which is invariant under changes of coordinates in the source and isometries in the target $$(u,\frac{a_{20}}{2}u^2+\frac{a_{30}}{6}u^3+\frac{1}{2}v^2+o(4)(u),\frac{b_{20}}{2}u^2+\frac{b_{30}}{6}u^3+\frac{b_{12}}{2}uv^2+\frac{b_{03}}{6}v^3+o(4)(u,v)).$$ Here $\kappa_a=a_{20}$ and $c'(0)=\pm\sqrt{-a_{20}}$, therefore $$(a^\pm)'''(0)=b_{30}-3b_{12}a_{20}\pm b_{03}(-a_{20})^{3/2}.$$ Since $b_{03}\neq 0\neq a_{20}$, $(a^+)'''(0)$ and $(a^-)'''(0)$ cannot be zero at the same time. Suppose it is $(a^+)'''(0)$ which is different from zero, this means that the function $a^+$ has an inflection point at 0 and changes sign when when $u$ goes from negative to positive. This implies that $\widetilde c^+$ also changes sign when $u$ goes from negative to positive. In fact, $\det(f_u(q),f_{uu}(q),f_{vv}(q))=0$ implies $b_{20}=0$, and so $\widetilde c^+(u)=f(u,c^+(u))$ has an inflection point at the origin.
\end{proof}

\begin{rem}
Most of the proof above is valid for any singularity such that $\Delta_p$ is a half-line. This includes all the fold singularities in Mond's list or most non-degenerate frontal singularities. However, the value of
$(a^\pm)'''(0)$ in \eqref{eq:apmppp}
may vary from one singularity to another. The criterion for hyperbolic points is always the same, however, for inflection points it may vary as the examples below suggest. At inflection points, the curves of intersection of $M$ with the osculating plane have contact order higher than two with one of the axis of coordinates of the plane, however the type of contact depends on $(a^\pm)'''(0)$ and the following derivatives, so a general statement would be too vague.
\end{rem}

\begin{ex}
\begin{itemize}
\item[i)] Consider the cuspidal edge given by $f(u,v)=(u,-u^2+v^2,u^2+v^3)$, the curvature parabola is parameterised by $(-2+2y^2,2)$ and $\kappa_a=-2<0$, so $p$ is a hyperbolic point. The intersection of $M$ with the osculating plane is given by the curves $(t,0,t^2+t^3)$ and $(-t,0,t^2+t^3)$, which meet tangentially at two local minima.
\item[ii)] Consider the cuspidal edges given by $f_1(u,v)=(u,-u^2+v^2,v^3)$ and $f_2(u,v)=(u,-u^2+v^2,u^3+uv^2+v^3)$. In both cases the curvature parabola is parameterised by $(-2+2y^2,0)$ and both have negative axial curvature, therefore the point $p$ is an inflection point. The intersection curves for the first case are given by $(t,0,t^3)$ and $(-t,0,t^3)$ which meet tangentially at inflection points of the curves. Notice that for $t$ small these curves lie on opposite half-planes. The intersection curves for the second case are given by $(t,0,3t^3)$ and $(-t,0,-t^3)$. These two curves also meet tangentially at inflection points but, differently from the first case, both curves always lie in the same half-plane for $t$ small.
\item[iii)] Consider the cuspidal edge given by $h(u,v)=(u,-u^2+u^3+v^2,u^3+v^3)$. Here $(a^-)'''(0)=0$. The intersection curves are given by $(t,0,t^3\pm (t^2-t^3)^{\frac{3}{2}})$, one of which has an inflection point at the origin.
\item[iv)] Consider the cuspidal cross-caps given by $g_1(u,v)=(u,-u^2+v^2,uv^3)$ and $g_2(u,v)=(u,-u^2+v^2,u^2+uv^3)$. The first case is an inflection point and the intersection curves of $M$ with the osculating plane are given by $(t,0,t^4)$ and $(-t,0,-t^4)$, which meet tangentially at a local minimum and a local maximum. The curves lie in different half-planes. The second case is a hyperbolic point and the intersection curves are given by $(t,0,t^2+t^4)$ and $(-t,0,t^2-t^4)$. These two curves meet at local minima and both lie in the same half-plane.
\end{itemize}
\end{ex}



\section{Relation of the Gaussian curvature with the axial curvature for certain fold singularities}

In order to consider the Gaussian curvature a unit normal vector field is needed. This is natural for frontal type singularities, but for other types of singularities
we need to use certain blow ups as in \cite{FukuiHasegawa,FukuiHasegawa2}. For the cross-cap singularity ($j^2f(0)\sim_{\mathcal{A}^{2}}(u,v^2,uv)$), Koenderink and Gauss-Bonnet type formulas have been obtained already (see \cite{FHS} and \cite{HHNSUY}). When $j^2f(0)\sim_{\mathcal{A}^{2}}(u,uv,0)$ the axial curvature is not bounded and when $j^2f(0)\sim_{\mathcal{A}^{2}}(u,0,0)$ the axial curvature is 0, so we consider only the case
$j^2 f(0)\sim_{\mathcal{A}^{2}}(u,v^{2},0)$. These singularities are called fold singularities by Mond in \cite{mond} and include the $S_k$, $B_k$ and $C_k$ singularities in his list, amongst others.

Let us assume
$j^2 f(0)\sim_{\mathcal{A}^{2}}(u,v^{2},0)$.
Then by a coordinate change on the source space and
by an action of $O(3)$ in the target space,
$f$ can be written in the following form.
For any $k\geq 1$,
\begin{align}
f(u,v)=&
\Bigg(
u, \dfrac{u^2}{2} a_0(u) + \dfrac{u^k v}{2} a_1(u) + \dfrac{v^2}{2} a_2(u, v),
\nonumber \\
\label{eq:formf}
&\hspace{2mm}
\dfrac{u^2}{2} b_0(u) + \dfrac{u^2 v}{2} b_1(u)
+ \dfrac{uv^2}{2} b_3(u) + \dfrac{v^3}{6} b_4(u, v)\Bigg),\quad(a_2(0,0)=1)
\end{align}
for some functions $a_0,a_1,a_2,b_0,b_1,b_3,b_4$.
See \cite{West,FukuiHasegawa,FukuiHasegawa2}.
We assume that $b_1(0)\ne0$, which includes $S_1$ or $B_k$ singularities in Mond's list, for example.

Let us set
$\Pi:\R\times S^1\to\R^2$ by
$$\Pi(r,\theta)=(r\cos\theta,r^2\cos\theta \sin\theta/2).$$
Then
$$
\Pi^*(f_{u}\times f_{v})
=
\dfrac{r^2}{2}\cos\theta\Big(
O(r^3),-b_1(0)\cos\theta+O(r^3),a_2(0,0) \sin\theta+O(r^3)\Big).
$$
Thus if we set
$$
\tilde \nu(r,\theta)=
\dfrac{\Pi^*(f_{u}\times f_{v})}
{r^2\cos\theta}
$$
and
$$
\nu(r,\theta)=\tilde\nu/\sqrt{\tilde\nu\cdot\tilde\nu},
$$
then the unit normal of $f$ is well-defined on the set
$(\R\times S^1;(r,\theta))$.

\begin{rem}
The assumption $b_1(0)\neq 0$ can be weakened by considering
$b_1(0)=\cdots =(b_1)^{(k)}(0)=0, (b_1)^{(k+1)}(0)\ne0$ instead.
Then the blow up should be changed to
$u=r\cos\theta,v=r^{k+1}\cos^k\theta \sin\theta/(k+1)!$ (\cite{FukuiHasegawa2}).
\end{rem}

We set
$E(r,\theta)=\Pi^*(f_{u}\cdot f_{u})$,
$F(r,\theta)=\Pi^*(f_{u}\cdot f_{v})$,
$G(r,\theta)=\Pi^*(f_{v}\cdot f_{v})$,
$L(r,\theta)=\Pi^*(f_{uu})\cdot \nu$,
$M(r,\theta)=\Pi^*(f_{uv})\cdot \nu$,
$N(r,\theta)=\Pi^*(f_{vv})\cdot \nu$,
and
$K=(LN-M^2)/(EG-F^2)$.
On this coordinate system, the Gaussian curvature $K$ can be computed as
$$
K=
\dfrac{
a_2(0,0) b_1(0)
\Big(a_0(0) b_1(0) \cos\theta-a_2(0,0) b_0(0) \sin\theta\Big)+O(r)
}
{r^4 \cos\theta\Big(b_1(0)^2 \cos^2\theta+a_2(0,0)^2 \sin^2\theta\Big)^2
/4+O(r^5)}.
$$
Since $a_2(0,0)=1$,
we can observe that the boundedness of the Gaussian curvature
is firstly controlled by the term
$$
b_1(0)\Big(a_0(0) b_1(0) \cos\theta-b_0(0) \sin\theta\Big).
$$
We set $\tilde K=b_1(0)\Big(a_0(0) b_1(0) \cos\theta
-b_0(0) \sin\theta\Big)$.

Since
the axial curvature $\kappa_a$ is $a_0(0)$, and
the umbilic curvature $\kappa_u$ is $b_0(0)$,
we have

\begin{prop}
Suppose that $j^2 f(0)\sim_{\mathcal{A}^{2}}(u,v^{2},0)$ and that $b_1(0)\neq 0$, then the boundedness of the Gaussian curvature depends on the term $$\tilde K=b_1(0)\Big(\kappa_a b_1(0) \cos\theta
-\kappa_u \sin\theta\Big).$$
\end{prop}

\begin{rem}
Koenderink type formulas relate the Gaussian curvature with the curvature of a section of the surface and the curvature of the apparent contour of a certain projection (\cite{koenderink}).
Let us set $\xi=(0,\cos\varphi,\sin\varphi)\in N_0M$ ($\sin\varphi\ne0$),
and set
$\pi_\xi(X)=X-(X\cdot\xi)\xi$, $\pi_\xi:\R^3\to\xi^\perp$.
A point $p$ is a singular point of $\pi_\xi\circ f$ if and
only if $\det(f_{u}(p),f_{v}(p),\xi)=0$.
We set $A(u,v)=\det(f_{u},f_{v},\xi)$. Then $A_{v}(0,0)=a_2(0,0)\sin\varphi\ne0$.
Thus there exists a function $v_1(u)$ such that
$A(u,v_1(u))=0$.
Then the contour of $f$ by $\pi_\xi$ is
$c(u)=\pi_\xi\circ f(u,v_1(u))$.
Since $A_{u}(0,0)=0$, $v_1'(0)=0$,
and $v_1''(0)=-A_{uu}(0,0)/A_{v}(0,0)=b_1(0)\cos\varphi/(a_2(0,0)\sin\varphi)$.
We have
$c(u)=u(1,0,0)+u^2(b_0(0) \cos\varphi-a_0(0)\sin\varphi)/2+O(u^3)$.
Thus the curvature of $c$ is
$$
-b_0(0) \cos\varphi+a_0(0) \sin\varphi+O(u).
$$
We set $\kappa_1=-b_0(0) \cos\varphi+a_0(0) \sin\varphi$.
Since
the axial curvature $\kappa_a$ is $a_0(0)$, and
the umbilic curvature $\kappa_u$ is $b_0(0)$,
we see
$\kappa_1=-\kappa_u \cos\varphi+\kappa_a \sin\varphi$,
and
$$
\tilde K=b_1(0)\Big(\kappa_a b_1(0) \cos\theta-
\dfrac{\kappa_a\sin\varphi-\kappa_1}{\cos\varphi} \sin\theta\Big).
$$
Thus, we can obtain a Koenderink type formula
if we can get $b_1(0)$ as a curvature of a slice
of $M$, however, this seems very difficult and we have not been able to do so.
\end{rem}

\subsection{Obstruction to being a frontal}

Here we consider the geometric meaning of $b_1(0)$.
Let $f=(f_1,f_2,f_3):\R^2\to\R^3$ be a germ, and $j^2f(0)\sim_{\mathcal{A}^{2}}(u,v^2,0)$.
Then we have a vector field $\eta$ such that
$\eta$ generates the kernel of $df$ on the singular set $S(f)$.
In this case, one can see $S(f)\subset \{v=0\}$ by a suitable coordinate change,
in particular, the regular set of $f$ is dense.
It is known that $f$ is a frontal if and only if the Jacobian ideal is principal
(generated by a single element) \cite[Lemma 2.3]{ishikawafrontal}.
Let $f$ be written in the form $(u,f_2(u,v),f_3(u,v))$.
Then we can choose $\eta=\partial_{v}$.
By the assumption $j^2f(0)\sim_{\mathcal{A}^{2}}(u,v^2,0)$ and $f_{v}=0$,
we have $f_{vv}\ne0$.
This means that one of $(f_2)_{v},(f_3)_{v}$ does not have a critical point
at $(0,0)$.
Let us assume that it is $(f_2)_{v}$, i.e.
$(f_2)_{vv}(0,0)\ne0$. Then there exists
a function $v(u)$ such that
$(f_2)_{v}(u,v(u))=0$ for all $u$.

\begin{prop}
The map $f$ is a frontal near $(0,0)$ if and only if
$(f_3)_{v}(u,v(u))=0$ for all $u$.
\end{prop}
\begin{proof}
Being a frontal or not does not depend on the choice of
coordinate systems,
we can change
the coordinate systems on the source and the target.
We may change $(u,v)$ so that $(u,v(u))$ is the $u$-axis.
Then $f$ has the form
$f=(u,a(u)+v^2\tilde f_2(u,v),\tilde f_3(u,v))$,
(where $\tilde f_3(u,v)=f_3(u,v-v(u))$).
By a coordinate change on the target,
we may assume
$f=(u,v^2\tilde f_2(u,v),\tilde f_3(u,v))$.
We may change $(u,v)$ so that
$f$ has the form
$f=(u,v^2,\tilde f_3(u,v))$.
$\tilde f_3$ can be written by $\tilde f_3(u,v)=b(u)+vc_1(u,v^2)+v^2c_2(u,v^2)$.
By a coordinate change on the target,
we may assume
$f=(u,v^2\tilde f_2(u,v),vc_1(u,v^2))$.
Then $f_{u}\times f_{v}$ is
$(2v,-2v(c_1)_{v}+c_1,2v^2)$.
Thus $f$ is a frontal if and only if $c_1$ can be divided by $v$,
namely $c_1(u,0)=0$. This is equivalent to $(\tilde f_3)_{v}(u,0)=0$.
This is equivalent to $(f_3)_{v}(u,v(u))=0$ for all $u$.
\end{proof}

Taking $f$ written by \eqref{eq:formf}, we see that a necessary condition that $(f_3)_{v}(u,v(u))=0$
is $b_1(0)=0$.

\begin{coro}
Consider $f$ as in \eqref{eq:formf}, then $b_1(0)\neq 0$ implies that $f$ is not a frontal.
\end{coro}

We define the first obstruction of frontality $\kappa_f$ as
$\kappa_f=b_1(0)$.

\end{document}